\documentclass[12pt,a4paper]{article}
\usepackage{amsmath,amsfonts,amssymb,amsthm}
\usepackage{color}
\usepackage{graphicx}
\newcommand{\xb}{{\boldsymbol{x}}}
\newcommand{\yb}{{\boldsymbol{y}}}
\newcommand{\Gb}{{\boldsymbol{G}}}

\newcommand{\Xb}{{\boldsymbol{X}}}
\newcommand{\vb}{{\boldsymbol{v}}}

\newcommand{\ab}{{\boldsymbol{a}}}

\newcommand{\Vb}{{\boldsymbol{V}}}

\newcommand{\etab}{{\boldsymbol{\eta}}}

\newcommand{\supp}{\operatorname{supp}}
\newcommand{\card}{\operatorname{card}}

\newcommand{\RE}{\operatorname{RE}}

\newcommand{\C}{\mathbb{C}}

\newcommand{\R}{\mathbb{R}}

\theoremstyle{plain}%
\newtheorem{lemma}{Lemma}
\newtheorem{theorem}{Theorem}
\newtheorem{definition}{Definition}
\newtheorem{corollary}{Corollary}

\begin{document}
%%%%%%%%%%%%%%%%%%%%%%%%%%%%%%%%%%%%%%%%%%%%%
	
\title{Predicting sparse circle maps from their dynamics}

\author{Felix Krahmer, Christian K\"uhn and Nada Sissouno}
\date{}
\maketitle
%%%%%%%%%%%%%%%%%%%%%%%%%%%%%%%%%%%%%%%%%%%%%

\begin{abstract}
The problem of identifying a dynamical system from its dynamics is of great importance for
many applications. Recently it has been suggested to impose sparsity models for improved
recovery performance. In this paper, we provide recovery guarantees for such a scenario.
More precisely, we show that ergodic systems on the circle described by sparse trigonometric
polynomials can be recovered from a number of samples scaling near-linearly in the sparsity.

\end{abstract}
%%%%%%%%%%%%%%%%%%%%%%%%%%%%%%%%%%%%%%%%%%%%%

\section{Introduction}
%%%%%%%%%%%%%%%%%%%%%%%%%%%%%%%%%%%%%%%%%%%%%

Many processes in science and engineering are naturally modeled by
nonlinear dynamical systems or, in the discrete situation, by iterated
maps~\cite{Strogatz}. While typically an exact analytic expression of the 
system is not known a priori, it is very important to find approximative 
formulas, e.g., for the prediction of a potentially catastrophic 
bifurcation~\cite{KuehnCT2}. A natural approach to find the approximate
analytical expression is that a vector $\yb\in\R^M$ of $M>0$ noisy observations 
of the system is given and assumed to be connected to the unknown parameter 
vector $\ab\in\mathbb{C}^N$ via a linear model
\begin{equation}\label{eq:linmodel}
\yb = \Gb \ab + \etab\,,
\end{equation}
where $\Gb$ is a $M\times N$ measurement matrix that encodes the relation between 
the parameter vector $\ab$ and the dynamical system and $\etab$ is the
noise vector of i.i.d.~Gaussian noise. The goal is then to recover $\ab$ from $\yb$. 
In order for such models to be sufficiently expressive to cover meaningful families 
of examples, one typically works with a large number of degrees of freedom and, 
consequently, can only expect recovery if a large number of observations of the 
system is available. To reduce the model complexity and hence the required sample 
complexity, it is natural to work with additional structural assumptions on the 
dynamical system and hence on the coefficient vector $\ab$. In particular, it has 
been proposed to impose sparsity constraints on $\ab$, that is, $\ab$ can be 
well-represented via only few vectors in a suitable representation system known 
a priori such as the monomial basis (see, e.g., \cite{BrPrKu16, BrPrKu13,WaETAL11}). 
Such models can be advantageous, as under certain assumptions, one can employ compressed 
sensing techniques \cite{FoRa13}, for example one can attempt efficient recovery of 
$\ab$ from undersampled information via $\ell_1$-minimization. This strategy is known 
to work if the sampling points are chosen at random. Hence, it seems natural to 
conjecture that this theory also may work within the context of ergodic dynamical 
systems~\cite{Walters}. Indeed, without ergodicity the domain can be separated into 
invariant regions and we cannot expect to describe a region we have not visited within
the available data. Yet, an ergodic system effectively has asymptotic behavior that 
resembles many stochastic features, which should then compensate the randomness in the
samples mentioned above. 

The diffficulty, however, is that the limiting probability distribution that will 
approximately describe the asymptotic behavior of the sampling procedure is the 
invariant measure of the unknown dynamical system and hence directly relates to the 
unknown quantity of interest. Consequently, the existing theory of compressed sensing 
does not apply without significant simplifying assumptions.

For example, \cite{ScTrWa18} assumes that the behavior can be observed for a large 
number of realizations of the dynamical system with uniformly distributed initializations. 
In this way, the initializations determine the sampling distributions, and the asymptotics 
of the dynamical system do not play a major role. Thus, the compressed sensing theory 
applies and implies exact recovery of dynamical system that have sparse representations 
with respect to Legendre polynomials.\medskip

In this paper we study the arguably more realistic, yet considerably more difficult 
situation, when the observations of the dynamical system are taken only from a single 
trajectory. We work in the same framework discussed above, that is, we aim to recover an ergodic 
dynamical system that has an (approximately) sparse representation in a basis
$\{\Phi_k\}_{k=1}^N$ for large $N\in\mathbb{N}$. We assume, that we observe the system in 
a statistically invariant state, that is, the measurements that are generated by the system 
are distributed according to an unknown invariant measure $\nu$ of the system. 

Note that there is an important conceptual difference to the compressed sensing 
literature: While the basis $\{\Phi_k\}_{k=1}^N$ is assumed to be known, the sampling 
measure $\nu$ is not only unknown, but even depends on the function of interest. As one 
is interested in recovering arbitrary functions $f$ from a class of candidates, one needs 
a recovery guarantee that holds uniformly with respect to the associated class of invariant 
measures. 

As a matter of fact, the basis functions $\Phi_k$ cannot be orthonormal with respect to all 
these candidate measures $\nu$. That is, in general, the covariance structure of random row 
selection model will not be the identity, we encounter so-called \emph{anisotropic 
measurements}. Such measurements have first been analyzed in \cite{RuZh13} establishing the 
so-called \emph{restricted eigenvalue (RE)} condition (see below). Later, an improved analysis 
using the golfing scheme has been provided in \cite{KuGr14}, but in this work, we focus on the 
original approach. Recovery guarantees under the assumption that $\Gb$ in \eqref{eq:linmodel} 
has the RE condition have been derived in in \cite{GeBu09}. The algorithm of choice used for 
recovery is either the $\ell_1$ penalized least square estimator, called Lasso
\cite{Ti96}, or the Dantzig selector \cite{CaTa07}. These methods have been fairly well understood. 
For example, prediction loss bounds are known \cite{BiRiTs09} when the number of variables is
much larger than the size of observations.

The goal shall hence be to show the restricted eigenvalue condition for the sampling matrix 
under consideration. In this paper, we provide the first example of a meaningfully rich class 
of dynamical systems, which all lead to matrices with the restricted eigenvalue condition. As 
we will explain and illustrate by numerical simulations, this will make them identifiable from 
a reduced number of observations via compressed sensing techniques.\medskip

\textit{Acknowledgments:} The authors would like to thank the German Science Foundation (DFG) for 
support via the SFB/TR109 ``Discretization in Geometry and Dynamics''. CK also acknowledges
partial support via a Lichtenberg Professorship granted by the VolkswagenStiftung and support
of the EU within the TiPES project funded the European Unions Horizon 2020 research and 
innovation programme under grant agreement No.~820970.

%We will determine assumptions on the invariant measures $\nu$ of the dynamical system 
%and on the basis $\{\Phi_k\}_{k=1}^N$ 

%These are methods where 

%%%%%%%%%%%%%%%%%%%%%%%%%%%%%%%%%%%%%%%%%%%%%
\section{Statement of problem}
%%%%%%%%%%%%%%%%%%%%%%%%%%%%%%%%%%%%%%%%%%%%%
We consider a dynamical system given by an iterated circle map
\begin{align}\label{eq:iterative_map}
	x_{k+1} = f(x_k)\quad\text{for }
			x_k\in [0,1]/(0\sim 1)=\mathbb{S}^1\,,
\end{align}
for some non-linear function $f:\mathbb{S}^1\rightarrow\mathbb{S}^1$. 
We assume that $f$ is can be represented as sparse trigonometric
polynomial and, thus, can be well approximated by a linear combination of
just few elements of the basis $\{\Phi_n\}_{n=1}^N$ with
\[
\Phi_n(x) \;=\; e^{2\pi i\,nx}
\]
for $N\gg 1$. That is,
there exist at most $s\ll N$, indices $n_1, \dots, n_s \in \{1, \dots, N\}$, 
and coefficients
$a_{n_k}\in \mathbb{C}$, $k\in \{1, \dots, s\}$, such that
\begin{align}\label{eq:jth_comp}
	f(x) \approx \sum_{k=1}^s a_{n_k}\cdot \Phi_{n_k}(x)\,.
\end{align}
Note that the indices $n_k$ of the active basis functions are not assumed to
be known a priori, which makes the reconstruction a very difficult nonlinear 
problem.

%To simplify matters, we
%focus only on an arbitrary fixed component and, therefore,
%skip the index $i$ correponding to the component or coordinate functions $f_i$.

The goal is to determine the coefficients in \eqref{eq:jth_comp} using
a minimum number of observations of the system.
The hope is that, due to the sparse structure \eqref{eq:jth_comp}, it will
only require the observation of $M\ll N$ states
$x_{j_1+1}=f(x_{j_1}), \dots, x_{j_M+1}=f(x_{j_M})$
of the dynamical system.
This corresponds to solving the underdetermined linear system
\begin{align}\label{eq:sys_of_lin_eq}
\left(\begin{matrix}
x_{j_1+1}\\\vdots\\x_{j_M+1}
\end{matrix} \right)
=
\left(\begin{matrix}
 \Phi_1(x_{j_1}) & \Phi_2(x_{j_1}) & \dots
	& \Phi_N(\xb_{j_1})
\\
\vdots & & & \\
 \Phi_1(x_{j_M}) & \Phi_2(x_{j_M}) & \dots
	& \Phi_N(x_{j_M})
\end{matrix} \right)
\left(\begin{matrix}
a_1\\\vdots\\a_N
\end{matrix} \right)\,
\end{align}
for each coordinate direction of the system.
Setting $ y_m:=x_{j_m+1}$ for $m=1,\dots,M$,
$\yb := (y_m)_{m=1}^M$, $\ab := (a_n)_{n=1}^N$, and
$ \Gb$ is a matrix with entries $G_{n,m}:= \Phi_n(x_{j_m})$
we can rewrite the problem in matrix notation to get $\yb = \Gb \ab$. 
If we assume that the measurements are noisy, we get the noisy linear 
model \eqref{eq:linmodel}.

We consider dynamical systems given by $f$ on $\mathbb{S}^1$ that possess 
a unique invariant probability measure $\nu$. We are going to assume that 
this invariant measure is absolutely continuous with respect to the Lebesque 
measure $\lambda$, i.e., 
\begin{equation}
\label{eq:Lebsac}
d \nu=h\,d\lambda
\end{equation}
for some density function $h$, which has full support in $\mathbb{S}^1$. We 
briefly comment on the size of the class
of circle maps, which admit such a unique measure with a density. If $f$ is 
continuous, there always exists an invariant measure for $f$ by applying the 
classical Krylov-Bogulyubov theorem~\cite{BogoliubovKrylov}. If $f$ is a 
homeomorphism with irrational rotation number~\cite{DeMeloVanStrien}, then it
is uniquely ergodic, which just means that there exists a unique $f$-invariant
probability measure $\nu$ with full support. If $f\in C^1$ is an 
orientation-preserving diffeomorphism, $f'$ is absolutely continuous, $(\ln f')'$ is in
$L^p(\mathbb{S}^1)$ for some $p>1$, and there is a suitable bound for the rotation
number, then one may prove~\cite{KatznelsonOrnstein} that the invariant 
measure is always absolutely continuous with respect to Lebesgue measure.
Of course, the conditions stated last are only sufficient, not necessary, 
and the class of circle maps which satisfy~\eqref{eq:Lebsac} is a lot larger;
for more details on circle maps see~\cite{DeMeloVanStrien}. In summary, the 
class of circle maps we consider is quite large and very robust as the conditions 
are open conditions in suitable function space topologies.

The covariance matrix of the basis $\{\Phi_n\}_{n=1}^N$ with respect to the measure 
$\nu$ is denoted by  $\Vb_\nu$ with entries
\begin{align}\label{eq:Cov_nu}
V_{\nu;j,k}:= \int_{\mathbb{S}^1}\Phi_j(x)\overline{\Phi_k(x)}\,d\nu(x)=
\langle \Phi_j,\,\Phi_k\rangle_\nu\,,
\end{align}
where $\langle\cdot,\cdot\rangle_\nu$ denotes the $L_2(\nu)$ inner product.
Note that in contrast to basically all other works on the topic, the covariance 
structure is directly connected to the quantity of interest and hence unknown.

%In case of the Lebesgue measure we skip the index $\nu$.
%It is clear that w.r.t the Lebesgue measure we have
%\begin{align*}
%&\langle \Phi_j,\,\Phi_k\rangle=\int_{0}^{1} e^{2\pi i\,jx}e^{-2\pi i\,kx} dx %= \delta_{j,k}
%=
%\begin{cases}
%%0 & \text{if }j\neq k,\\
%1 & \text{if }j=k,
%\end{cases}\\
%&\max_{x}|\Phi_k(x)|\le 1\quad \text{for all }k\in\mathbb{Z}.
%\end{align*}
%Thus, they form a bounded orthonormal system and the recovery results from 
%\cite{FoRa13} can be applied. In this case, exact recovery can be guaranteed 
%with high probability using
%$\ell_1$-minimization. If the invariant measure of the circle map is not the 
%measure, we have anisotropic measurements and an alternative method needs to be chosen.

%%%%%%%%%%%%%%%%%%%%%%%%%%%%%%%%%%%%%%%%%%%%%
\section{Prediction loss estimates for circle maps}
%%%%%%%%%%%%%%%%%%%%%%%%%%%%%%%%%%%%%%%%%%%%%
In general, we aim to find
the sparse solution of the problem \eqref{eq:sys_of_lin_eq} or \eqref{eq:linmodel} since
we assume that only a few $a_n$ are nonzero. In the situation that at most $s$ entries are nonzero a vector
$\ab\in\mathbb{C}^N$ is called \emph{$s$-sparse}, that is, if
\begin{align*}
\|\ab\|_0:=\card(\supp(\ab))\le s\,,
\end{align*}
where the \emph{support} of a vector is given by
$\supp(\ab):=\{n\in\{1,\dots,N\}:\,a_n\neq 0\}$. The set of all
$s$-sparse vectors is denoted by
\begin{align*}
\Sigma_s=\{\ab\in\mathbb{C}^N:\,\|\ab\|_0\le s\}\,.
\end{align*} 
The compressed sensing problem is given by the search for a sparse
solution of a system of linear equations, like
\eqref{eq:sys_of_lin_eq} in the noiseless situation. Theoretically,
this solution can be determined by the \emph{$\ell_0$-minimization}
\begin{align*}
\min_{\ab} \|\ab\|_0 \quad \text{s.t. }\yb=\Gb\ab\,.
\end{align*}
Since this is, in general, an NP-hard problem %\cite{},
alternative methods
are needed and are widely studied (cf. \cite{FoRa13}).
%The \emph{basic pursuit} or
%\emph{$\ell_1$-minimization} already given in \eqref{eq:l1-min}
%is one of those methods. 
The question is under which conditions
an $s$-sparse solution of this minimization exists; for detailed
discussion see for example \cite{FoRa13} or \cite{BoETAL15a}.

For recovery we work with the Lasso, which for some fixed parameter  $\lambda>0$ is given by the
optimization problem
\begin{equation}\label{eq:Lasso}\tag{L}
\ab_L :=\arg \min_{\ab}\bigg\{ \frac{1}{M}\|\yb-\Gb\ab\|_2^2+2\lambda \|\ab\|_1\bigg\}.
\end{equation}

Obviously, the recovery properties of these estimators strongly depend on the measurement matrix $\Gb$. In \cite{BiRiTs09} the restricted eigenvalue condition is introduced which is one of the weakest conditions
(see \cite{GeBu09}) in order to guarantee desirable properties of the two estimators.

\begin{definition}\label{def:RE}
Let $s_0$ be some integer with $0<s_0<N$ and $p$ a positive number.
A matrix $\Xb$ satisfies the \emph{$\RE(s_0,p,\Xb)$ condition with parameter
$\kappa(s_0,p,\Xb)$} if for any $\vb\neq 0$
\begin{equation}\label{eq:RE}
\frac{1}{\kappa(s_0,p,\Xb)}:=
\min_{\substack{I\subset \{1,\dots,N\},\\ |I|\leq s_0}}\;\min_{\|\vb_{I^c}\|_1\leq p\|\vb_{I}\|_1}\frac{\|\Xb \vb\|_2}{\|\vb_I\|_2}>0\,.
\end{equation}
Here $\vb_I$ denotes the subvector of $\vb$ restricted to the indices given by $I\subset \{1,\dots,N\}$.
\end{definition}

We will also need the concept of the smallest $s$-sparse eigenvalue.

\begin{definition}\label{def:minEV}
For $s\leq N$, we define the \emph{smallest $s$-sparse eigenvalue} of a matrix ${\Xb}$ as
\begin{equation}
\rho_{\min}(s,\Xb):=\min_{\substack{\vb\in\C^N, \,\vb\neq 0\\s-\text{sparse}}}
\frac{\|\Xb\vb\|_2^2}{\|\vb\|_2^2}\,.
\end{equation}
\end{definition}

In general, for $M\ll N$ the equation $\eqref{eq:linmodel}$ does not have a unique
solution. Nevertheless, as discussed in \cite{BiRiTs09} we have a unique $s$-sparse solution 
\begin{equation}
\ab^\ast=\{\ab\in\C^N:\,\yb = \Gb \ab + \etab\}\cap \Sigma_{s}\,.
\end{equation}
The solutions ${\ab}_L$ of the Lasso has the property (cf. \cite[Appendix B]{BiRiTs09}) that chosing $p=3$ with high probability
\begin{equation*}
\|({\ab}_L-\ab^\ast)_{I_0^c}\|_1\le p\|({\ab}_L-\ab^\ast)_{I_0}\|_1\,,
\end{equation*}
where $I_0:=\supp(\ab^\ast)$.
Based on this property and the $\RE$ condition the following bounds on the rates of convergence of Lasso
can be obtained. 

We will show the following result.

\begin{theorem}\label{thm:main}
Let $\eta_m$, $m=1,\dots,M$, be independent $\mathcal{N}(0,\sigma^2)$
random variables with $\sigma^2>0$, let $N\geq 2$, and let $1\leq s\leq N$.
Consider a circle map $f$, i.e., $0<f(x)<1$ for all $x\in \mathbb{S}^1$, which is an $s$-sparse trigonometric polynomial, i.e., a linear combination of only $s$ complex exponentials with integer frequencies of at most $N$. 

Consider the iterative dynamical system on the circle described by the map $f$ and 
assume that the density function $h$ of its invariant measure $\nu$ satisfies 
\[h(x)\geq \xi_h\geq \sqrt{C_1\frac{s}{N}}
\]for all $x\in\mathbb{S}^1$. Let the sample size $M$ satisfy
\begin{align}\label{eq:M-bound}
M &\geq \frac{C_2\,s \cdot\log(N)}{\xi_h^{3/2}}\cdot\log\bigg(\frac{C_2\,s\cdot\log(N)}{\xi_h^{3/2}}\bigg)
\end{align}
and choose 
\[
\lambda=4\sigma\sqrt{\frac{\log N}{M}}\,.
\]

Then we have, with probability at least $1-N^{-1}$, that the Lasso returns an estimate $a_L$, which satisfies
\begin{align}
\|\ab_L-\ab^\ast\|_1&\leq C_3\xi_h^{-1} \,s\sigma\sqrt{\frac{\log N}{M}}\,.
%\|\Gb(\ab_L-\ab^\ast)\|_2^2&\leq C_3\xi_h^{-1} \,s_0\,\sigma^2\log(N).
\end{align}
Here $C_1,\,C_2,\,C_3$ are some absolute constants.
%\end{enumerate}
Consequently, the reconstructed circle map $f_L$ satisfies
\begin{align}
\|f_L-f\|_{A(\mathbb{T})}&\leq C_3\xi_h^{-1} \,s\sigma\sqrt{\frac{\log N}{M}},
%\|\Gb(\ab_L-\ab^\ast)\|_2^2&\leq C_3\xi_h^{-1} \,s_0\,\sigma^2\log(N).
\end{align}
where $A(\mathbb{T})$ denotes the Wiener algebra.
\end{theorem}

\section{Proof of Theorem 1}
Our proof is based on the following Theorem of \cite{BiRiTs09}.
\begin{theorem}\label{thm:loss-LDS}
Let $\eta_m$, $m=1,\dots,M$, be independent $\mathcal{N}(0,\sigma^2)$
random variables with $\sigma^2>0$ and let $1\leq s_0\leq N$, $M\geq 1$, and $N\geq 2$. Let all diagonal elements of ${\Gb}^\ast{\Gb}/{M}$ be equal to $1$ and for $A>2\sqrt{2}$ choose 
\[
\lambda=A\sigma\sqrt{\frac{\log N}{M}}\,.
\]
 If $\Gb$ fulfills $\RE(s_0,3,M^{-1/2}\Gb)$, then we have, with probability at least $1-N^{1-A^2/8}$, that the Lasso estimator satisfies
\begin{align}
\|\ab_L-\ab^\ast\|_1&\leq 16\kappa^2(s_0,3,M^{-1/2}\Gb) \,s\lambda\,.
\end{align}
\end{theorem}

The following conditions are satisfied by the invariant measure $\nu$ of the dynamical system and the basis $\{\Phi_n\}_{n=1}^N$:

\begin{enumerate}
\item[(C1)] The density function $h$ of the invariant measure $\nu$ is bounded from below by a constant $\xi_h>0$.

\item[(C2)] The basis $\{\Phi_n\}_{n=1}^N$ is orthogonal with respect to the Lebesque measure and satisfies that
\begin{itemize}\itemindent -10pt
\item[i)] $\|\Phi_n\|_\infty<1$ for all $n\in\{1,\dots,N\}$,
\item[ii)] $\|\Phi_n\|_2^2=1$ for all $n\in\{1,\dots,N\}$.
\end{itemize}

\end{enumerate}

With these conditions we can show that the $\RE$ condition of the square root of the covariance matrix ${\Vb}_{\nu}$ of the basis $\{\Phi_n\}_{n=1}^N$
is satisfied independent of the choice of $s_0$ and $p$.

\begin{lemma}\label{le:RE}
Let $\{\Phi_n\}_{n=1}^N$ and $\nu$ satisfy (C1) and (C2). Let $s_0$ be some integer with $0<s_0<N$ and $p$ a positive number. Then $\Vb_{\nu}^{{1}/{2}}$ satisfies the $\RE(s_0,p,\Vb_{\nu}^{{1}/{2}})=\RE(\Vb_{\nu}^{{1}/{2}})$ condition with parameter
\begin{equation*}
\kappa(s_0,p,\Vb_{\nu}^{{1}/{2}})=\kappa(\Vb_{\nu}^{{1}/{2}})=\frac{1}{\sqrt{\xi_h}}\,.
\end{equation*}
\end{lemma}

\begin{proof}
With (C1) and (C2) we get
\begin{align*}
\|\Vb^{1/2}_\nu \vb\|_2^2 & = \vb^\ast \Vb\vb= \Bigg\langle\sum_{j=1}^N v_j\Phi_j,\,
\sum_{k=1}^N v_k\Phi_k\Bigg\rangle_\nu
= \int_{\mathbb{S}^1} \bigg|\sum_{j=1}^N v_j\Phi_j(\xb)\bigg|^2h(\xb)\,d\xb\\
&\geq \xi_h\sum_{j,k=1}^Nv_j\bar{v}_k\langle\Phi_j,\,\Phi_k\rangle\geq
\xi_h\|\vb\|^2_2\,.
\end{align*}
Insertion into \eqref{eq:RE} shows that
$\kappa(s_0,p,\Vb_{\nu}^{{1}/{2}})={\xi_h}^{-1/2}$ independent of the choice of $s_0$ and $p$.
\end{proof}

Analogously, we can show similar results for the smallest $s_0$-sparse eigenvalue.

\begin{lemma}\label{le:minEV}
Let $\{\Phi_n\}_{n=1}^N$ and $\nu$ satisfy (A1) and (A2). Let $s$ be some integer with $0<s<N$. Then the smallest $s$-sparse eigenvalue of $\Vb_{\nu}^{{1}/{2}}$ is given by
\begin{equation*}
\rho_{\min}(s,\Vb_{\nu}^{{1}/{2}})=\rho_{\min}(\Vb_{\nu}^{{1}/{2}})= \xi_h\,.
\end{equation*}
\end{lemma}

With these two lemmata combined with \cite[Theorem 8]{RuZh13} we get the following statement about the measurement matrix.

\begin{corollary}\label{cor:RE-G}
Let $0<\delta<1$, $0<s_0<N$, and $p>0$. Set
\begin{align*}
\ell= s_0+s_0 \frac{16\,(3p)^2(3p+1)}{\sqrt{\xi_h}\delta^2}\,.
\end{align*}
Assume $\ell\leq N$. Let the sample size $M$ satisfy
\begin{align}%\label{eq:M-bound1}
M &\geq \frac{C\,\ell\cdot\log(N)}{\xi_h\delta^2}\cdot\log\bigg(\frac{C\,\ell\cdot\log(N)}{\xi_h\delta^2}\bigg)\,,
\end{align}
where $C$ is some absolute constant. 
Then, with probability $1-exp(-\frac{\delta \xi_h M}{6\ell})$,
the matrix $M^{-1/2}\Gb$ satisfies the
 the $RE(s_0,p,M^{-1/2}\Gb)$ condition with parameter
\[
0<\kappa\big(s_0,p,M^{-1/2}\Gb\big)\leq \frac{1}{\sqrt{\xi_h}\,(1-\delta)}\,.
\]
\end{corollary}

Theorem 1 follows as a consequence of Corollary \ref{cor:RE-G}, which is true for arbitrary choices of $s_0$ and $p$, and Theorem \ref{thm:loss-LDS} by setting $\delta=\frac{1}{2}$ and $p=3$.
The recovery guarantee for $f$ follows directly from the definition of the Wiener Algebra.
%%%%%%%%%%%%%%%%%%%%%%%%%%%%%%%%%%%%%%%%%%%%%%%%%%%%%%%%%%%%%%%%%%%%%
\bibliography{Lit_dynSys_compSens.bib}
%%%%%%%%%%%%%%%%%%%%%%%%%%%%%%%%%%%%%%%%%%%%%%%%%%%%%%%%%%%%%%%%%%%%%
\end{document}